\documentclass[12pt]{article}
\usepackage{amssymb}
\usepackage[dvips]{epsfig}
\usepackage{times}
\usepackage{graphicx}

\makeatletter\@addtoreset {equation}{section}\makeatother

\newtheorem{theorem}{Theorem}
\newtheorem{lemma}{Lemma}

\newenvironment{proof}{
    \noindent {\it Proof.}}{\hfill$\Box$
}

\def\eps{\varepsilon}
\def\supp{{\rm support\,}}

\begin{document}

\title{\bf On the energy-minimizing steady states \\ of a thin film equation}

\author{Almut Burchard\\
University of Toronto\\
{\tt almut@math.utoronto.ca}\\[0.7cm]
Marina Chugunova\\
University of Toronto\\
{\tt chugunom@math.utoronto.ca}\\[0.7cm] 
Benjamin K. Stephens\\
University of Washington\\
{\tt benstph@math.washington.edu}\\[0.7cm]}

\date{March 9, 2010}

\maketitle

\bigskip
\bigskip

\begin{abstract}
Steady states of the thin film equation
$u_t+[u^3\,(u_{\theta\theta\theta} + \alpha^2 u_\theta -\sin\theta)]_\theta=0$
are considered on the periodic domain $\Omega = (-\pi,\pi)$.
The equation defines a generalized gradient flow for an
energy functional that controls the $H^1(\Omega)$-norm. The
main result establishes that there exists for each given mass a
unique nonnegative function of minimal energy.
This minimizer is symmetric decreasing about $\theta=0$.
For $\alpha<1$ there is a critical value for the mass
at which the minimizer has a touchdown zero.
If the mass exceeds this value, the minimizer is
strictly positive. Otherwise, it is supported
on a proper subinterval of $\Omega$ and meets the dry region at zero
contact angle.  A second result explores the relation between
strict positivity and exponential convergence for steady states.
It is shown that positive minimizers are locally
exponentially attractive, while the distance from
a steady state with a dry region cannot decay faster
than a power law.

\end{abstract}

\newpage

\section{Introduction}

Degenerate fourth order parabolic equations of the form
$$
u_t + \nabla \cdot (u^n \nabla \Delta u) + \mbox{lower order terms}=0
$$
are commonly used to model the evolution of thin liquid films
on the surface of a solid.   Here, $u(x,t)$ describes the thickness
of the fluid at time $t$ at the point $x$,
the fourth derivative term models surface tension,
and the exponent $n>0$
is determined by the boundary condition between the liquid and
the surface of  the cylinder. A particularly interesting
case is $n=3$, which corresponds to a
``no-slip'' boundary condition.

In this paper, we study the equation
\begin{equation} \label{eq:PDE}
 u_t + \left[u^n\,(u_{\theta\theta\theta} +
\alpha^2\,u_\theta-\sin\theta) + \omega u \right]_\theta = 0\,,\quad
\theta \in \Omega = (-\pi, \pi)
\end{equation}
with periodic boundary conditions.
For $n=3$ and $\alpha=1$, this describes the evolution of a thin liquid
film on the outside of a horizontal cylinder that rotates slowly
about its axis, see Figure~1.
The film is assumed to be uniform  along the axis
of the cylinder, and its
thickness at time $t$ and angle~$\theta$ (measured from the bottom)
is given by the function $u(\theta,t)$.
In Eq.~(\ref{eq:PDE}), the first summand in the parentheses
models surface tension, the next term is
a correction due to the curvature of the cylinder, and the third term
models gravitational drainage.
The last term  models convection due to rotation. 
We have scaled the units of length and
time so that the surface tension and gravitational terms
appear with coefficient one; the coefficient $\alpha\ge 0$ is
a geometric constant, and $\omega$ is proportional to the
speed of rotation. Here, we will
study the the non-rotating cylinder with $\omega=0$.
We are mainly interested in the case where $n=3$ and $\alpha=1$, 
but find it illuminating to also  consider other values of $n$ and
$\alpha$.

\begin{figure}[ht]
\label{fig:cylinder}
\begin{center}
\includegraphics[height=5cm] {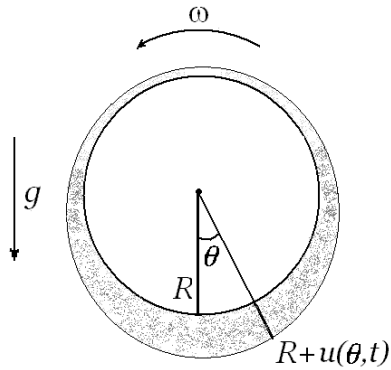}\\
\end{center}
\caption{\small Liquid film on the outer surface of a rotating horizontal
cylinder in the presence of gravity.}
\end{figure}

The model in Eq.~(\ref{eq:PDE}) (with $n=3$ and $\alpha=1$) was
first derived and studied by
Pukhnachov \cite{Pukh1, Pukh3}. Note that Pukhnachov uses slightly
different coordinates (with $\theta=\frac{3\pi }{2}$
at the bottom of the cylinder), and a different scaling (where
$\omega=1$).  The derivation relies on the lubrication approximation;
it assumes that the liquid film is very thin
compared to the radius of the cylinder, and that
the rotation is slow enough to neglect centrifugal forces.
Pukhnachov's model refines an earlier model
of Moffatt~\cite{Moffatt} by including surface tension.
The second order curvature term in the equation is reminiscent
of a porous-medium equation, but appears with the opposite sign,
resulting in a long-wave instability~\cite{B15}.
Interesting numerical and asymptotical
analysis along with numerous open questions can be found in
\cite{Benilov1, Benilov3, Kar}.

A fundamental technical problem in thin film equations
is well-posedness, i.e., to show that nonnegative initial data 
give rise to unique nonnegative solutions that depend continuously on
the data. The difficulty is that solutions of fourth order
parabolic equations generally do not satisfy maximum or
comparison principles, and linearization leads to semigroups
that do not preserve positivity. To give  a simple example,
the function $u(x,t) = 1 + t\cos(x)$, which  solves
Eq.~(\ref{eq:PDE}) in the linear case  $n=0$ with
$\alpha=1$ and $\omega=0$, ceases to be positive for $t>1$.
Thus, positivity of solutions of thin film equations
is a nonlinear phenomenon.
Other relevant problems in the area  concern long-term
behavior, finite-time blow-up, and the
interface between wet ($u>0$) and dry $(u=0)$ regions.
These problems have been studied rigorously in a vast
body of papers since the pioneering article of Bernis
and Friedman~\cite{BF}, see for
example~\cite{B2, BertPugh1996,B15,Report,Otto}
and references therein.  There is an even
larger literature
that studies the properties of physically relevant solutions
through asymptotic expansions, numerical analysis, and laboratory
experiments.

Bernis and Friedman proved that initial-value problems
for thin film equations
can be solved in suitable classes of nonnegative weak
solutions.  An important technical contribution was their use of an
{\em entropy functional} that decreases with time along
solutions. A few years later, the so-called $\alpha$-entropies
were independently discovered by Bertozzi and Pugh~\cite{BertPugh1996}
and by Beretta, Bertsch, and dal Passo~\cite{B2}.
Since then, other families of entropies have been
found~\cite{CDGJ,MJ,Laug}.  Entropy functionals are the
basis for results on short- and
long-term existence, positivity, finite speed of propagation,
regularity, blow-up, and the long-time behavior of solutions.

For $0<n\le 1$, well-posedness and convergence
to steady states have recently been established by treating
Eq.~(\ref{eq:PDE}) as a gradient flow on a space of 
measures endowed with the Wasserstein distance~\cite{MMS,Otto}, 
where the exponent $n$ appears as a mobility parameter~\cite{CLSS}.
These gradient flow techniques also take advantage
of energy and entropy dissipation. 
However, Wasserstein distances with mobility $n>1$ 
are not well understood, and well-posedness remains an open 
problem.

By a {\em solution} of Eq.~(\ref{eq:PDE})
we mean a nonnegative function
$u\in L^2\bigl((0,T), H^2(\Omega)\bigr)$ that satisfies
$$
\int_0^T\int_\Omega
\Bigl\{ u\phi_t -
(u_{\theta\theta} + \alpha^2\,u + \cos\theta)\bigl(u^n\phi_\theta\bigr)_\theta
+ \omega u \phi_\theta\Bigr\}\, d\theta dt = 0
$$
for every smooth test function
with compact support in $\Omega\times (0,T)$.
This agrees with the strongest notion of generalized solutions
from~\cite{BF,BertPugh1996}.
For a class of equations that includes Eq.~(\ref{eq:PDE})
with $n=3$, long-time existence of such solutions was recently
proved in \cite{Report}. These solutions are widely
believed to be unique.

Questions about steady states appear in many applications. When
do steady states exist, when are they uniquely determined by their mass, 
are they strictly
positive or do they exhibit dry regions, and under what conditions
are they stable? Do steady states attract all bounded solutions?
When can we expect exponential convergence?
Linearizations of Eq.~(\ref{eq:PDE})
about steady states were examined analytically and numerically
in \cite{ChKP09,John}.

For $\omega=0$, Eq.~(\ref{eq:PDE})
defines a generalized gradient flow for the {\em energy}
\begin{equation}
\label{eq:def-E}
\label{EF} E(u) = \frac{1}{2} \int_\Omega u_\theta^2 -
\alpha^2\,u^2\,d\theta -\int_{\Omega} u\,\cos\theta\,d\theta\, ,
\end{equation}
in the sense that $u_t =
\left[u^n\left(\frac{\delta E}
{\delta u}\right)_{\theta}\right]_\theta$.
Here, $\frac{\delta E}{\delta u}$ denotes the $L^2$-gradient
of $E$.  This implies the dissipation estimate
\begin{equation}\label{eq:dissipate}
\frac{d}{dt} E\bigl(u(\cdot, t)\bigr)
= - \int_\Omega
u^n \left(\frac{\delta E}{\delta u}\right)_{\!\theta}^2
\, d\theta
\le 0\,.
\end{equation}
The subject of this paper are the minimizers
of $E$ on the set of nonnegative $2\pi$-periodic functions of a given mass
$$
C_M = \left\{u\in H^1(\Omega) \ \Big\vert \
u\ge 0, \int_{\Omega} u\, d\theta = M \right\}
$$
and their role in the dynamics.
Note that for $\alpha=1$, $E$ is convex but
not strictly convex.  For $\alpha<1$, the functional
is strictly convex, and
has a unique critical point on $C_M$, which is a global minimizer.
For $\alpha>1$, it is  not convex, and we may
expect multiple critical points.

We will show that $E$ has a unique minimizer on $C_M$ for each
value of $\alpha$ and each mass $M>0$,
see Theorem~\ref{thm:min}.
These minimizers may have dry regions; in that case, the
contact angle of the fluid film is
zero, see the Figure~2.
In particular, our result establishes the existence of
zero contact angle steady states for Eq~(\ref{eq:PDE}) if
$M(1-\alpha^2)\le 2\pi$.
Our proof relies on symmetric decreasing rearrangements
and the first variation of the energy.

The minimizers are time-independent solutions
of Eq.~(\ref{eq:PDE}) with $\omega=0$.
Additional steady states may arise for $\alpha>1$ as saddle points
of the energy. For any value of $\alpha$ and $M$, there
is also a continuum of steady states
with have non-zero contact angles,
analogous to the steady states in~\cite{LaugPugh}, 
whose role in the evolution remains unclear.

We expect that for $\alpha\le 1$, the unique energy minimizer
should attract all solutions on $C_M$ as $t\to\infty$.
Unfortunately, in the absence of a proper well-posedness
theory, Lyapunov's theorem is not sufficient to support this
expectation. In Theorem~\ref{thm:decay} we provide
partial results in that direction.  If the energy-minimizing steady state
is strictly positive, then it exponentially attracts all
solutions in a neighborhood.  On the other hand, for $n>\frac{3}{2}$,
a steady state that has a dry interval
of positive length cannot be exponentially attractive.
In particular for $n=3$ and $\alpha=1$,
the distance between the solution and the minimizer
decays no faster than $t^{-\frac{2}{3}}$.
Our proof combines energy and entropy inequalities
in the spirit of~\cite{B2,BertPugh1996,CU,Tudorascu}.

All our results are easily adapted to the
long-wave stable case where the sign of the second
order term is reversed. In that case,
the energy-minimizing steady state is strictly positive and locally
exponentially attractive so long as $M(1+\alpha^2)>2\pi$.
For $M(1+\alpha^2)<2\pi$, the energy minimizer has a dry interval
of positive length, which it meets at zero contact angles.
This contrasts with a theorem of Laugesen and Pugh that
excludes zero contact angle steady states
for the corresponding thin film equation without the
sine term~\cite{LaugPugh}.

\section{Identification of energy minimizers}

We begin by showing that for every $M>0$ there exists a
function $u$ with mass $M$ that minimizes the energy.
The first lemma provides the necessary global
bounds on  the functional.

\begin{lemma} [Lower bound on the energy.]
\label{lem:energy-bound}
$E$ is bounded from below and coercive on $C_{M}$.
\end{lemma}

\begin{proof}
Using that $u$ is nonnegative and has mean $\frac{M}{2\pi}$,
we estimate
$$
\int_{\Omega} u^2 dx \le M \,\|u \|_{L^{\infty}}\,,
\quad
||u||_\infty \le \frac{M}{2\pi} + \sqrt{\pi}||u_\theta||_{L^2}\,.
$$
It follows that
\begin{equation} \label{eq:energy-bound}
E(u) \ge  \frac{\pi}{2}\left(
||u||_\infty - \frac{M}{2\pi}(1+\alpha^2)\right)^2 -
\frac{M^2}{4\pi}\alpha^2(2+\alpha^2) -M\,.
\end{equation}
which is clearly bounded below. A similar estimate
shows that $E$ grows quadratically as $||u||_{H^1}\to\infty$.
\end{proof}

\bigskip The lemma implies that minimizing sequences are
bounded in $H^1$. Passing to a subsequence,
we can construct  a minimizing sequence $\{u_j\}_{j\ge 1}$
that converges weakly in $H^1$ and strongly in $L^2$
to some function $u$ in $C_M$. Since $E$ is weakly
lower semicontinuous on $H^1$, $u$
is the desired minimizer.  We next describe some properties of
the minimizers.

\begin{lemma} [Symmetry.]
\label{lem:symmetry} Minimizers of $E$
on $C_M$ are symmetric decreasing about $\theta=0$.
\end{lemma}

\begin{proof}
For $u\in C_M$, let  $u^\#$ be the unique symmetric decreasing
function of $\theta$ that is equimeasurable to~$u$.
Classical results about symmetric decreasing rearrangement
ensure that $u^\#\in C_M$, and that
\begin{equation}\label{eq:rearrange}
||u^\#||_{L^2}=||u||_{L^2}\,,\quad
||u^\#_\theta||_{L^2}\le ||u_\theta||_{L^2}\,,\quad
\int_\Omega u^\# \cos\theta \, d\theta \ge
\int_\Omega u \cos\theta \, d\theta \,,
\end{equation}
see~\cite{PSz}. It follows that
$$
E(u^\#)\le E(u)\,.
$$
If $u$ is a minimizer, then $E(u^\#)=E(u)$, and in particular,
the third inequality in
Eq.~(\ref{eq:rearrange}) must hold with equality. Since
the cosine is {\em strictly} symmetric  decreasing,
this forces $u$ to be symmetric decreasing as well~\cite[Theorem 3.4]{LL}.
\end{proof}

\bigskip
The Euler-Lagrange equation for the minimizer is given by
\begin{equation}\label{eq:EL}
u_{\theta\theta} +\alpha^2 u +\cos\theta
= \lambda \quad \ \mbox{on}\ \{\theta\in\Omega\mid u(\theta)>0\}\,,
\end{equation}
where $\lambda$ is a Lagrange multiplier associated with the
mass constraint.
By considering the first variation of $E$ with respect to nonnegative
functions that need not vanish outside the support of $u$, we see that
$$
u_{\theta\theta} +\alpha^2 u + \cos\theta \le \lambda
\quad \ \mbox{on}\ \Omega
$$
as a distribution. This suggests that the first derivative of
a minimizer should vanish at the boundary of its support,
i.e., the film meets the surface of the cylinder at zero contact angle.
The next lemma confirms this suspicion.

\begin{lemma} [Zero contact angle.]
\label{lem:zero}
Let $u$ be a minimizer of $E$ on $C_M$.
If $u$ has its first positive zero at $\theta=\tau$, then
$u_\theta(\tau_-)=0$ and $u_{\theta\theta}(\tau_-)>0$.
In particular, $u\in {\cal C}^{1,1}(\Omega)$.
\end{lemma}

\begin{figure}[ht]
\label{fig:steadystates}
\begin{center}
\includegraphics[height=5cm] {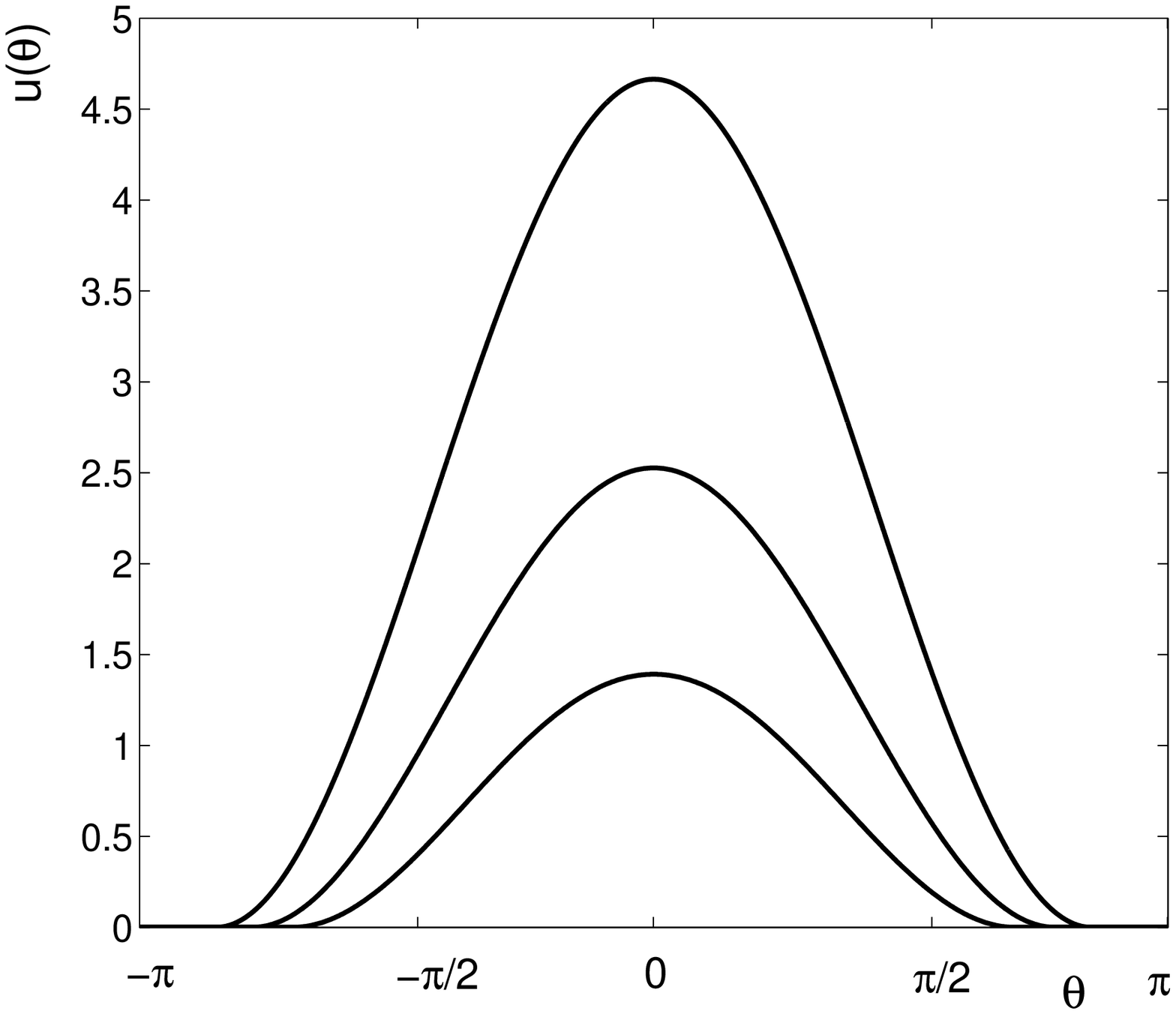}
\quad\quad \quad\quad
\includegraphics[height=4.5cm] {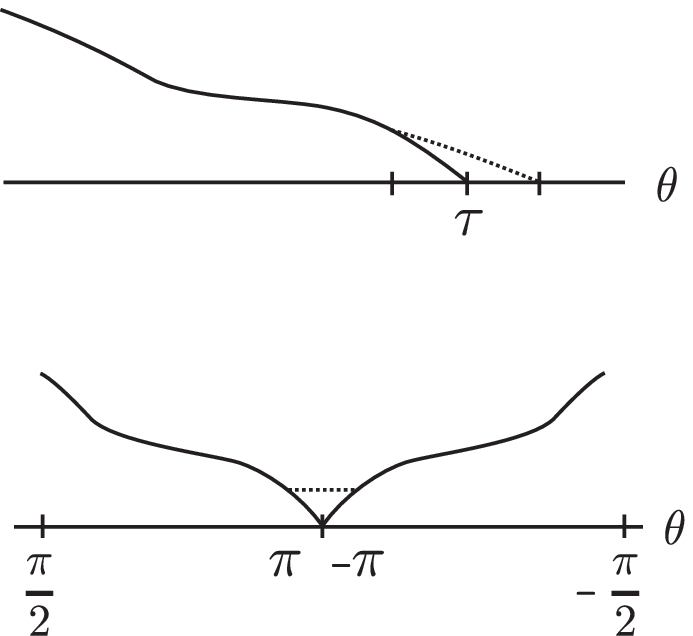}
\end{center}
\caption{\small
On the left: Steady
states for $\alpha=1$ and initial data
$ u_0 = 0.5, 1, 2$.  As the mass is getting smaller, the
minimizer becomes more concentrated, and as mass goes to infinity
the support tends to $[-\pi, \pi]$.
On the right:
Energy-decreasing variations
resulting from non-zero contact angles.
}
\end{figure}

\begin{proof} By Eq.~(\ref{eq:EL}), $u$ has one-sided derivatives of
arbitrary order at $\tau$, and $u_\theta(\tau_-)\le 0$.
Suppose that $u_\theta(\tau_-)< 0$. We will modify $u$ to
construct a new valid competitor with lower energy.
If $\tau<\pi$, set $v(\theta)=u(\phi(\theta))$, where
$\phi:\Omega\to\Omega$ is the bi-Lipschitz map defined by
$$
\phi(\theta)=\left\{\begin{array}{ll}
\frac{1}{2}(\theta+\tau-\eps)\,,\quad &\tau-\eps\le\theta\le\tau+\eps\,,
\\[0.1cm]
2\theta-\tau-2\eps\,,& \tau+\eps\le \theta\le \tau + 2\eps\,,  \\[0.1cm]
\theta\,,& \mbox{otherwise}\,.
\end{array}\right.
$$
Choose $\eps>0$ small enough so that
$u$ vanishes on $[\tau,\tau+2\eps]$.  Then $v$
vanishes on $[\tau+\eps,\tau+2\eps]$.
The difference between the leading terms in the energy integrals
is given by
$$
\frac{1}{2} \int_\Omega  v_\theta^2\, d\theta
-\frac{1}{2} \int_\Omega u_\theta^2\, d\theta
= \frac{1}{2} \int_{\tau-\epsilon}^{\tau + \epsilon} v_\theta^2 \, d\theta
- \frac{1}{2} \int_{\tau-\epsilon}^\tau u_\theta^2 \, d\theta
= -\frac{\eps}{4} u_\theta(\tau_-)^2 + O(\eps^2)\,,
$$
see Figure~2, top right.
Since the remaining terms contribute only corrections
of order $O(\eps)^2$ to the energy difference, it follows that
$$
E(v)-E(u) = -\frac{\eps}{4} u_\theta(\tau_-)^2 + O(\eps)^2\,.
$$
If $\tau=\pi$, the same estimate holds (by the symmetry of $u$)
for $v(\theta)= u(\min\{|\theta|,\pi-\frac{1}{4}\eps\})$,
see Figure~2, bottom right.
In either case, the mass of $v$ is
$M'= \int_\Omega v \,  d\theta = M + O(\eps)^2$.
We finally set  $w(\theta)= \frac{M}{M'} v(\theta)$,
which has the correct mass and satisfies
$E(w)=E(v)+O(\eps^2)<E(u)$ for $\eps$ sufficiently small,
a contradiction. We conclude that $u_\theta(\tau_-)=0$,
proving the first claim.

To prove the second claim, we analyze the sign
of the first non-vanishing left derivative of $u$
at $\tau$.  Clearly, $u_{\theta\theta}(\tau_-)\ge 0$.
Suppose that $u_{\theta\theta}(\tau_-)=0$.
Differentiating Eq.~(\ref{eq:EL}), we obtain
$u_{\theta\theta\theta}(\tau_-)=\sin\tau\ge 0$.
Since $\tau$ is the first positive zero of $u$,
the derivative $u$ from the left
cannot be positive, and so $\theta=\pi$ is the
only possibility.  Differentiating
once more, we obtain $u_{\theta\theta\theta\theta}(\pi_-)=-1$,
which is the wrong sign for $u$ to have a minimum at $\pi$. It
follows that $u_{\theta\theta}(\tau_-)>0$, as claimed.

\end{proof}

\bigskip
We are now ready to compute the minimizers of $E$ on $C_M$
explicitly.
A particular solution of the
Euler-Lagrange equation in Eq.~(\ref{eq:EL}) is given by
\begin{equation}
\label{eq:u-part}
u^0(\theta)= \left\{ \begin{array}{ll}
-\frac{1}{2} \theta\sin\theta\,,\quad & \alpha=1\,,\\[0.1cm]
\frac{1}{1-\alpha^2}
\left( \cos\theta -\frac{1+\alpha^2}{2\alpha}\cos(\alpha\theta)\right)\,,
\quad & \alpha\ne 1\,,
\end{array}\right.
\end{equation}
and the general solution can be represented as
\begin{equation}\label{eq:EL-sol}
u(\theta) =A\cos(\alpha \theta) + B \sin(\alpha\theta) +
\frac{\lambda}{\alpha^2} + u^0(\theta)\,.
\end{equation}
The following theorem summarizes our results.
The statement is illustrated in Figure~3.

\begin{figure}[ht]
\label{fig:mass-tau}
\begin{center}
\includegraphics[height=5cm] {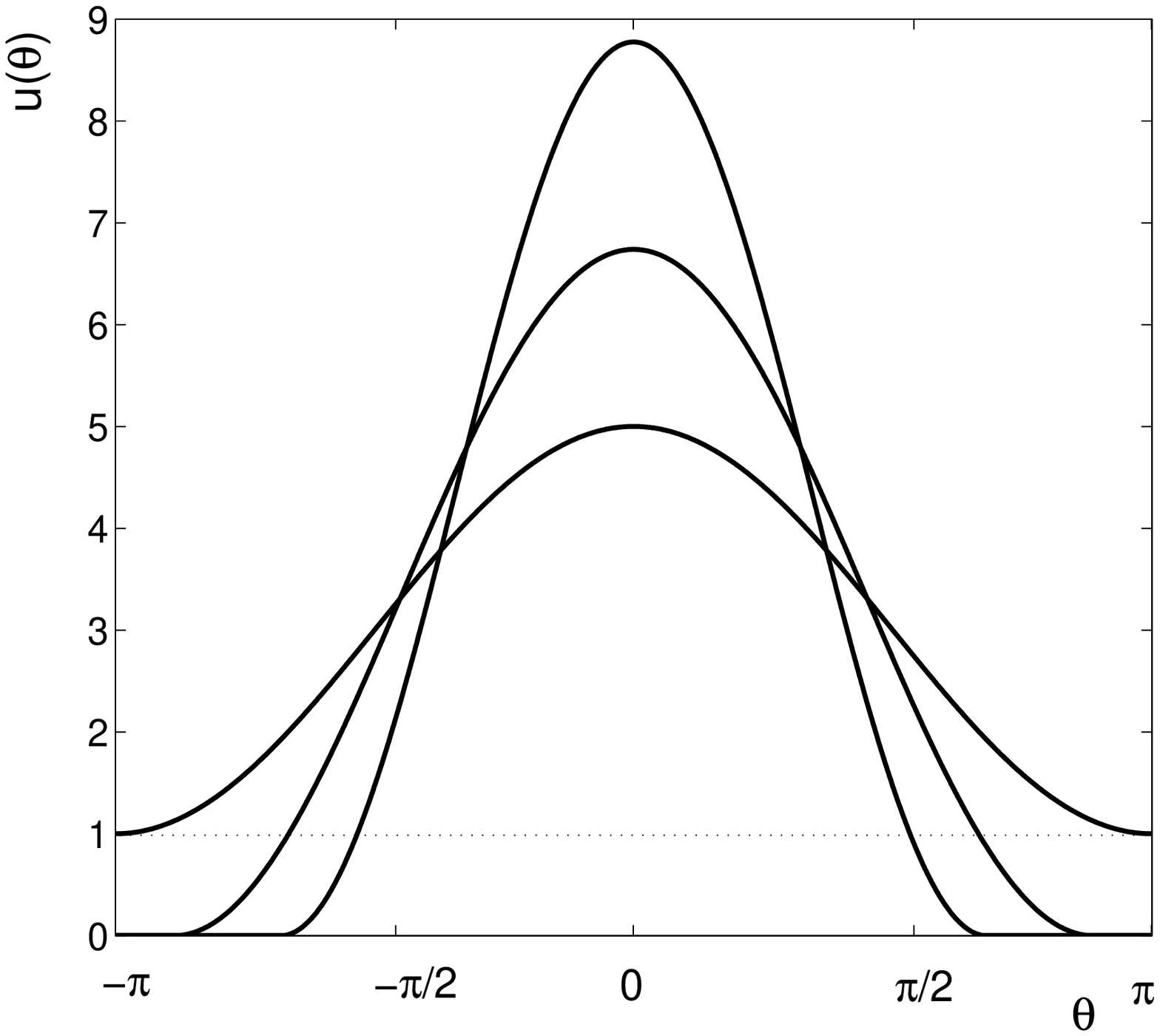}
\quad \quad \quad \quad
\includegraphics[height=5cm] {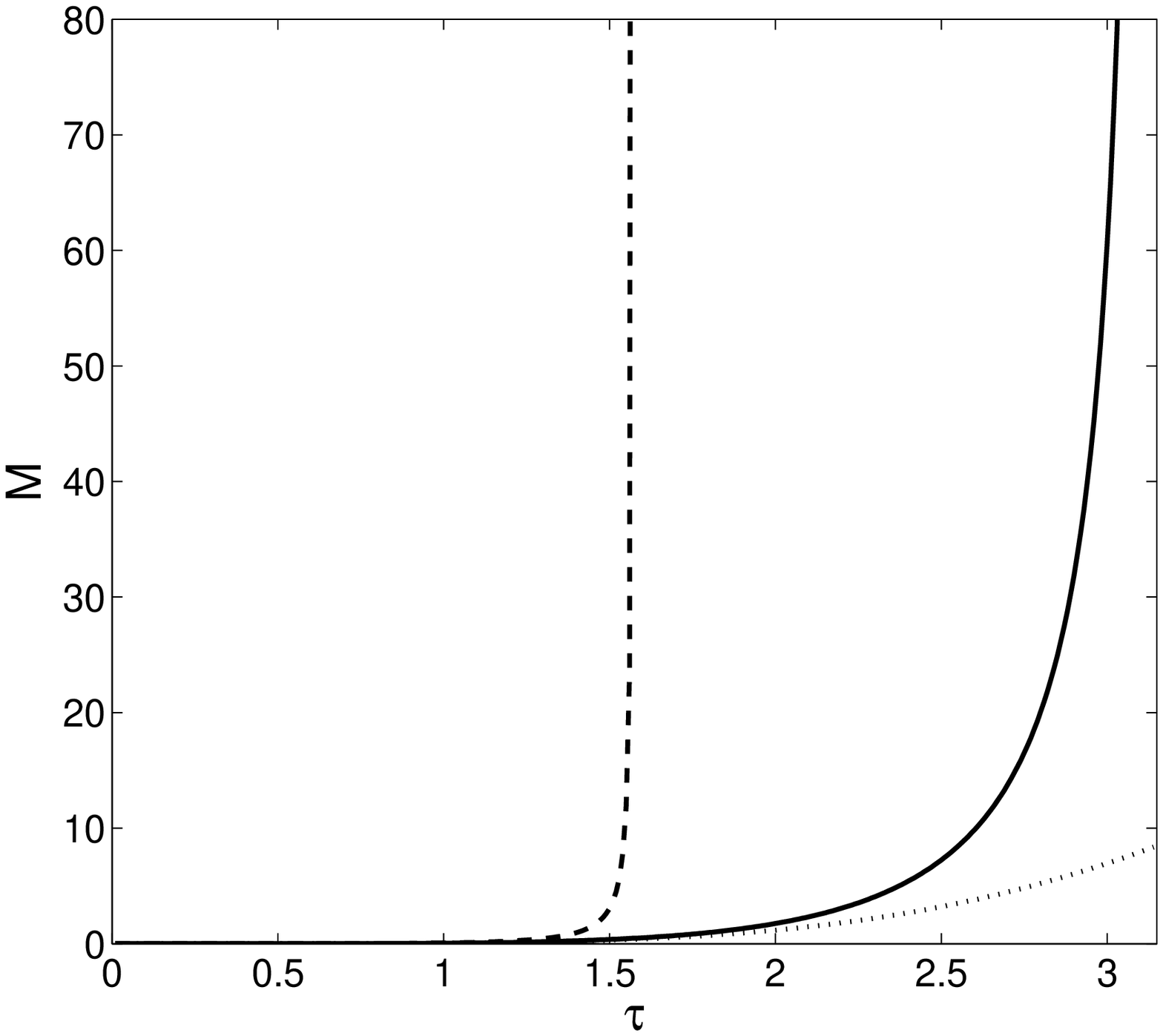}
\end{center}
\caption{\small On the left: Numerical steady states
for $\alpha=0.5, 1, 2$ with initial
data $u_0=3$.
On the right: Mass versus half length of the compact support for $\alpha
= 2, 1, 0.5$.}
\end{figure}

\begin{theorem}[Description of the energy minimizers.]
\label{thm:min}
Let $E$ be the energy functional
in Eq.~(\ref{eq:def-E}).
For each $M>0$, $E$ has a unique nonnegative minimizer
of mass $M$.  The minimizer is
strictly symmetric decreasing on its support.
It is of class ${\cal C}^{1,1}$
and depends continuously on $M$ in ${\cal C}^{1,1}$.
It increases with $M$ in the sense that
for any pair of minimizers $u_1,u_2$ of mass $M_1,M_2$,
$$
M_1< M_2\ \Longrightarrow\ u_1(\theta)< u_2(\theta)\,,
\quad \theta\in\supp(u_1)\,.
$$
\begin{itemize}
\item
If $M(1-\alpha^2)> 2\pi$, the minimizer is strictly positive and given by
\begin{equation}\label{eq:u-positive}
u(\theta)=\frac{M}{2\pi} + \frac{1}{1-\alpha^2}\cos\theta\,;
\end{equation}
\item if $M(1-\alpha^2)<2\pi$, the minimizer
is given by
\begin{equation}\label{eq:u}
u(\theta) = A(\tau)\bigl(\cos(\alpha\theta)-\cos(\alpha\tau)\bigr)
+ u^0(\theta)-u^0(\tau)\,,\quad |\theta|\le\tau
\end{equation}
for some $\tau$ with $\max\{\alpha,1\}\,\tau<\pi$, and vanishes for
$|\theta|\ge \tau$. The coefficient is determined by
$A(\tau)=\frac{u^0_\theta (\tau)}{\alpha\sin(\alpha\tau)}$, where
$u^0$ is the special solution from Eq.~(\ref{eq:u-part});
\item if $M(1-\alpha^2)=2\pi$, then
Eq.~(\ref{eq:u}) for $\tau=\pi$
coincides with Eq.~(\ref{eq:u-positive}), and
$u(\theta)=\frac{1+\cos\theta}{1-\alpha^2}$.
\end{itemize}
\end{theorem}

\begin{proof}  Fix $M>0$.  By Lemma~\ref{lem:energy-bound},
there exists a minimizer $u$ of mass $M$,
and by Lemma~\ref{lem:symmetry},
it is symmetric decreasing about $\theta=0$.
If the positivity constraint is not active,
then Eq.~(\ref{eq:EL-sol}) holds for all $\theta\in\Omega$.
Since the minimizer is smooth, periodic, and has mass $M$,
we conclude that $\alpha<1$ and Eq.~(\ref{eq:u-positive})
holds. In order for $u$ to
be nonnegative and symmetric decreasing
we must have $M(1-\alpha^2)\ge 2\pi$.  In that region,
$u$ is clearly strictly increasing in $\theta$.

If, on the other hand, the positivity constraint is active, then
the minimizer $u$ is positive on some interval
$(-\tau,\tau)$ and vanishes for $|\theta|\ge \tau$.
By Lemma~\ref{lem:zero}, $u\in {\cal C}^{1,1}(\Omega)$ and
$u_\theta(\pm\tau)=0$.
On its support, $u$ is given by Eq.~(\ref{eq:EL-sol}).
Since $u$ and $u^0$ are even, $B=0$.  The Dirichlet
condition at $\tau$ allows to
eliminate $\lambda$, the Neumann condition determines $A$,
and we find that Eq.~(\ref{eq:u}) holds.
If $\max\{\alpha,1\}\,\tau<\pi$, we claim that
$u$ is indeed nonnegative, symmetric
decreasing in $\theta$, and strictly increasing with $\tau$.
To see this, we differentiate Eq.~(\ref{eq:u}), and use
Lemma~\ref{lem:zero} to obtain
$$
\frac{dA}{d\tau}\cdot  \alpha\sin\alpha\tau =
- u_{\theta\tau}(\tau;\tau)= u_{\theta\theta}(\tau;\tau)>0\,.
$$
By the chain rule, and using once more
that $u_\theta(\tau;\tau)=0$, we have
$$
u_\tau (\theta;\tau) = \frac{dA}{d\tau}
\cdot \bigl( \cos(\alpha \theta)-\cos(\alpha\tau) \bigr)>0
\quad\mbox{for}\ |\theta|<\tau\,.
$$
Since $u$ vanishes identically when
$M=0$, this confirms that $u$ is
positive and strictly symmetric decreasing for $|\theta|<\tau$.
We use that $u_\theta(\tau;\tau)=0$ to compute
$$
\frac{dM}{d\tau} = \frac{dA}{d\tau}\int_{-\tau}^\tau
\cos(\alpha\theta)-\cos(\alpha\tau)\,d\theta>0\,,
$$
and infer that  we can solve for $\tau=\tau(M)$ as a
strictly increasing
smooth function of $M$. By the chain rule and the inverse function
theorem,
$$
\frac{d}{d M}
u(\theta;\tau(M)) = \frac{\cos(\alpha\theta)- \cos(\alpha\tau)}{
\int_{-\tau}^\tau \cos(\alpha\theta')- \cos(\alpha\tau)\, d\theta'}
>0
\,,
$$
establishing the desired continuity and monotonicity of $u$ with
respect to $M$ in the range where Eq.~(\ref{eq:u}) is valid
and $\max\{\alpha,1\}\,\tau<\pi$.

We need to determine the ranges where Eq.~(\ref{eq:u-positive})
and (\ref{eq:u}) hold.  For $\alpha<1$, $E$ is strictly convex.
If $M\ge \frac{2\pi }{1-\alpha^2}$, the function defined
by Eq.~(\ref{eq:u-positive}) is nonnegative and
provides the unique minimizer of $E$ on $C_M$.
If $M< \frac{2\pi}{1-\alpha^2}$, the positivity constraint is
active, and we compute from Eq.~(\ref{eq:u})
that $M\to 0$ as $\tau\to 0$ and $M\to \frac{2\pi}{1-\alpha^2}$
as $\tau\to \pi_-$. Continuous dependence on $M$
follows, since Eq.~(\ref{eq:u}) agrees with
Eq~(\ref{eq:u-positive}) at
$M=\frac{2\pi}{1-\alpha^2}$.

For $\alpha=1$, $E$ is convex, but not strictly convex on $C_M$.
The positivity constraint is active,
because $E$ is not bounded below without it;
for instance, $E\left(\frac{M}{2\pi} +t\cos\theta\right) = -\pi t$.
Note that we  must have $\tau<\pi$,
because the particular solution
$u^0(\theta)= -\frac{1}{2}\theta\sin\theta$
from Eq.~(\ref{eq:u-part}) cannot be continued as
a differentiable periodic function  across $\theta=\pi$,
in violation of Lemma~\ref{lem:zero}.
It is easy to check from Eq.~(\ref{eq:u})
that $M\to 0$ as $\tau\to 0$, and $M\to\infty$ as $\tau\to \pi$.

For $\alpha>1$, the energy is a non-convex quadratic
function on $C_M$, and hence the positivity constraint is active
and $u$ is given by Eq.~(\ref{eq:EL-sol})
on some interval $(-\tau,\tau)$.
Upon closer inspection of Eq.~(\ref{eq:EL-sol}),
we see that $\alpha\tau<1$,
since otherwise $u$ fails to be
symmetric decreasing.  Since $M\to 0$ as $\tau\to 0$
and $M\to\infty$ as $\tau \to \alpha^{-1}\pi$,
the  theorem follows.
\end{proof}

\section{Convergence to minimizers}

In this section, we will prove a lower bound on the speed at
which solutions of Eq.~(\ref{eq:PDE}) with $\omega=0$ can converge
to critical points on the boundary of the positive cone.
We establish this bound for two classes of
solutions: For strictly positive, classical solutions
when $n>\frac{3}{2}$, and for the strong generalized solutions
constructed for $n=3$ in~\cite[Section 3]{Report}.

Our bound uses the {\em entropy method}, applied to the functional
\begin{equation} \label{eq:def-S}
S(u)=\int_ \Omega u^{-\beta}\,d\theta\,,
\end{equation}
where $\beta =n-\frac{3}{2}$.  Strictly speaking,
$S$ is not an entropy  for
Eq.~(\ref{eq:PDE}), because  it may increase as well
as decrease along solutions. One of the reasons is that
the porous medium term $\alpha^2 (u^nu_\theta)_\theta$
appears in Eq.~(\ref{eq:PDE}) with the unfavorable sign.  Still, the
standard entropy methods yields a useful differential
inequality for $S$.

\begin{lemma} [Entropy inequality for classical solutions]
\label{lem:entropy}
Fix $n>\frac{3}{2}$ and let $S$ be given by Eq.~(\ref{eq:def-S})
with $\beta=n-\frac{3}{2}$. For every strictly positive classical
classical solution $u$ of Eq.~(\ref{eq:PDE}),
there exist constants $S_0$ and $K_0$ such that
\begin{equation}\label{eq:entropy}
S(u(\cdot, t)) \le S_0 + K_0t\,.
\end{equation}
\end{lemma}

\begin{proof} Let $S_0=S(u(\cdot, 0))$ and $E_0=E(u(\cdot, 0))$
be the initial values of the entropy and energy, and set
$c_n= \bigl(n\!-\!\frac{3}{2}\bigr)\bigl(n\!-\!\frac{1}{2}\bigr)$.
We will show  that
$$
\frac{d}{dt} S(u(\cdot, t))\le K_0\,,
$$
where
$$
K_0= c_n\left\{
\frac{M\alpha^2}{4}\left[ \frac{M}{2\pi}(1+\alpha^2)
+\left( \frac{2(E_0+M)}{\pi}
  + \frac{M^2}{4\pi^2}\alpha^2(2+\alpha^2)\right)^{\frac{1}{2}}
\right]^{\frac{1}{2}}+ 2\sqrt{M\pi}\right\}\,.
$$
The claim then follows by integrating along the solution.

To see the differential inequality, we use the 
Eq.~(\ref{eq:PDE}) and integrate by parts,
\begin{equation}
c_n^{-1} \frac{dS(u)}{dt}
= \int_\Omega u^{-\frac{1}{2}} u_\theta u_{\theta\theta\theta}\, d\theta
+ \alpha^2 \int_\Omega u^{-\frac{1}{2}}
u_\theta^2\,d\theta
- \int_\Omega u^{-\frac{1}{2}}u_\theta\sin\theta \, d\theta\,.
\label{eq:S-dot-proof}
\end{equation}
The first summand in Eq.~(\ref{eq:S-dot-proof}) we integrate again by
parts,
$$
\int_{\Omega}
u^{-\frac{1}{2}}u_{\theta} u_{\theta\theta\theta}\, d\theta
= -\int_\Omega u^{-\frac{1}{2}}u_{\theta\theta}^2 \, d\theta
+\frac{1}{2} \int_\Omega u^{-\frac{3}{2}}
u_\theta^2 u_{\theta\theta}\, d\theta
=: -A + \frac{1}{2} B\,,
$$
and integrate by parts once more to see that
$B =
\frac{1}{2} \int u^{-\frac{5}{2}}u_\theta^4\, d\theta$.
After collecting terms, the first summand becomes
$$
-A+B -\frac{1}{2} B = -\int_\Omega u^{-\frac{1}{2}}
\left( u_{\theta\theta} - \frac{1}{2}u^{-1}u_{\theta}^2\right)^2\,
d\theta
 = -4\int_\Omega u^{\frac{1}{2}}
\left((u^{\frac{1}{2}})_{\theta\theta}\right)^2\,d\theta\,.
$$
The second summand in Eq.~(\ref{eq:S-dot-proof}) we
also integrate by parts,
$$
\alpha^2\int_\Omega  u^{-\frac{1}{2}} u_\theta^2\, d\theta =
2 \alpha^2 \int_\Omega u (u^{\frac{1}{2}})_{\theta\theta}\, d\theta\,,
$$
and the third summand we rewrite as
$$ -\int_\Omega  u^{-\frac{1}{2}} u_\theta\sin\theta\, d\theta=
2\int_\Omega u^{\frac{1}{2}}\cos\theta\, d\theta\,.
$$
Inserting these identities into Eq.~(\ref{eq:S-dot-proof})
and completing the square we arrive at
\begin{eqnarray}
\label{eq:S-dot}
\frac{dS(u)}{dt} &=& c_n \left\{ - \int_\Omega u^{\frac{1}{2}}
\left( 2( u^{\frac{1}{2}})_{\theta\theta}-
\frac{\alpha^2}{2} u^{\frac{1}{2}}\right)^2\,d\theta
+ \frac{\alpha^4}{4} \int_\Omega u^{\frac{3}{2}}\,d\theta
+2\int_\Omega u^{\frac{1}{2}}\cos\theta\, d\theta\right\}\\
\nonumber &\le &
c_n \left\{ \frac{\alpha^4}{4} M
\sqrt{||u||_\infty} + 2\sqrt{M\pi}\right\}\,.
\end{eqnarray}
The claim follows from Eq.~(\ref{eq:energy-bound})
of Lemma~\ref{lem:energy-bound},
\end{proof}

\bigskip The entropy inequality in Eq.~(\ref{eq:entropy})
holds also for many types of weak solutions of Eq.~(\ref{eq:PDE})
that are obtained as limits of classical solutions
of suitable regularizations. We demonstrate this for the
strong generalized solutions in the case $n=3$
that were recently constructed by Chugunova, Pugh, and
Taranets in~\cite[Theorem 2]{Report}.
For $\eps>0$, consider the regularized equation
\begin{equation} \label{eq:PDE-eps}
 u_t + \left[f_{\epsilon}(u)\,(u_{\theta\theta\theta} +
\alpha^2\,u_{\theta} - \sin\theta) \right]_\theta =
0\,,\quad \theta \in \Omega\,,
\end{equation}
where $f_\eps (z) = \frac{z^4}{|z| + \epsilon}$.
This equation has strictly positive classical solutions
positive initial data in $H^1(\Omega)$
(see~\cite[Lemma 3.4]{Report}), and the energy in Eq.~(\ref{eq:def-E})
is dissipated.

The key step is to extend the entropy method to the regularized
equation.  Set
\begin{equation}
\label{eq:def-S-eps}
  S_\eps(u) = \int_{\Omega} s_\eps(u)\, dx\,,
\end{equation}
where $s_\eps(z)=z^{-\frac{3}{2}}(1+\frac{3}{7}\eps z^{-1})$
is chosen so that $s_\eps''(z)f_\eps(z)= c_3 z^{-\frac{1}{2}}$.
Then $\frac{dS_\eps}{dt}$ satisfies the entropy identity
in Eq.~(\ref{eq:S-dot}) along solutions  of Eq.~(\ref{eq:PDE-eps}),
and hence
$$
S_\eps(u(\cdot, t))-S_\eps(u(\cdot, 0))\le K_0t\,,
$$
with the same constant as in Lemma~\ref{lem:entropy}.
As $\eps\to 0$ along a suitable subsequence, solutions of
Eq.~(\ref{eq:PDE-eps}) converge
uniformly to solutions of the original problem in
Eq.~(\ref{eq:PDE}), and the values
of the energy and entropy also converge
for all times,
provided they are finite at $t=0$.
Thus the entropy inequality in Eq.~(\ref{eq:entropy})
remains valid in the limit. This can be used to establish
additional regularity properties: Using a suitable subsequence
where $u_{\theta\theta}$ converges weakly
in $L^2(\Omega\times(0,T))$ one can show that for every $T>0$,
$$
\int_0^T \int_\Omega u^{\frac{1}{2}} \left( 2(
u^{\frac{1}{2}})_{\theta\theta}- \frac{\alpha^2}{2}
u^{\frac{1}{2}}\right)^2\,d\theta dt
\le S(u(\cdot, 0))+K_0T<\infty \,,
$$
and conclude that $\int_\Omega u^{\frac{1}{2}}
\left((u^{\frac{1}{2}})_{\theta\theta}\right) ^2\, d\theta $
is finite  for almost every $t>0$.

Our final result concerns the dynamics near the energy
minimizer, see Figure~4.

\begin{figure}[t]
\label{fig:converge}
\begin{center}
\includegraphics[height=5cm] {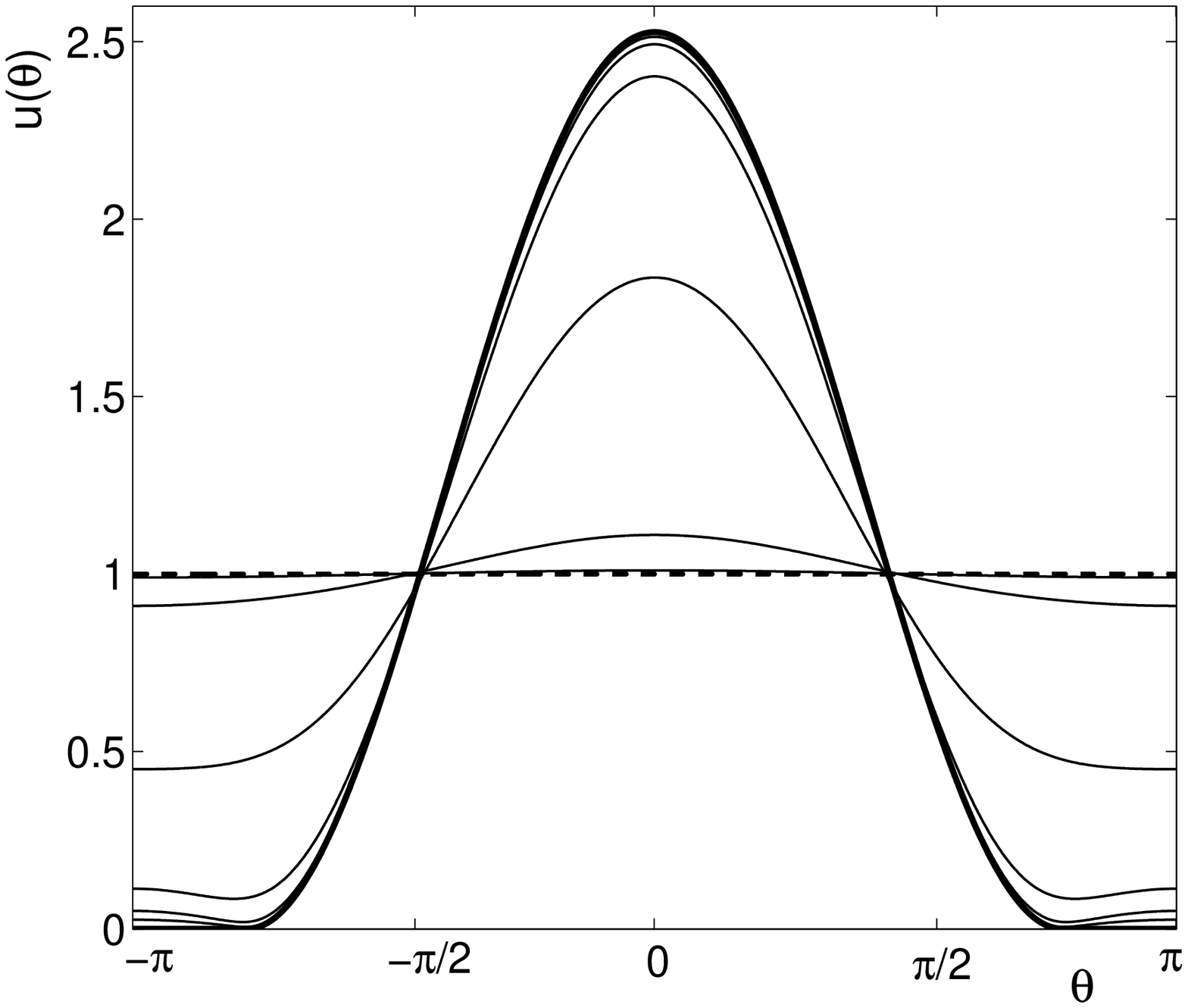}
\quad\quad
\quad\quad
\includegraphics[height=5cm] {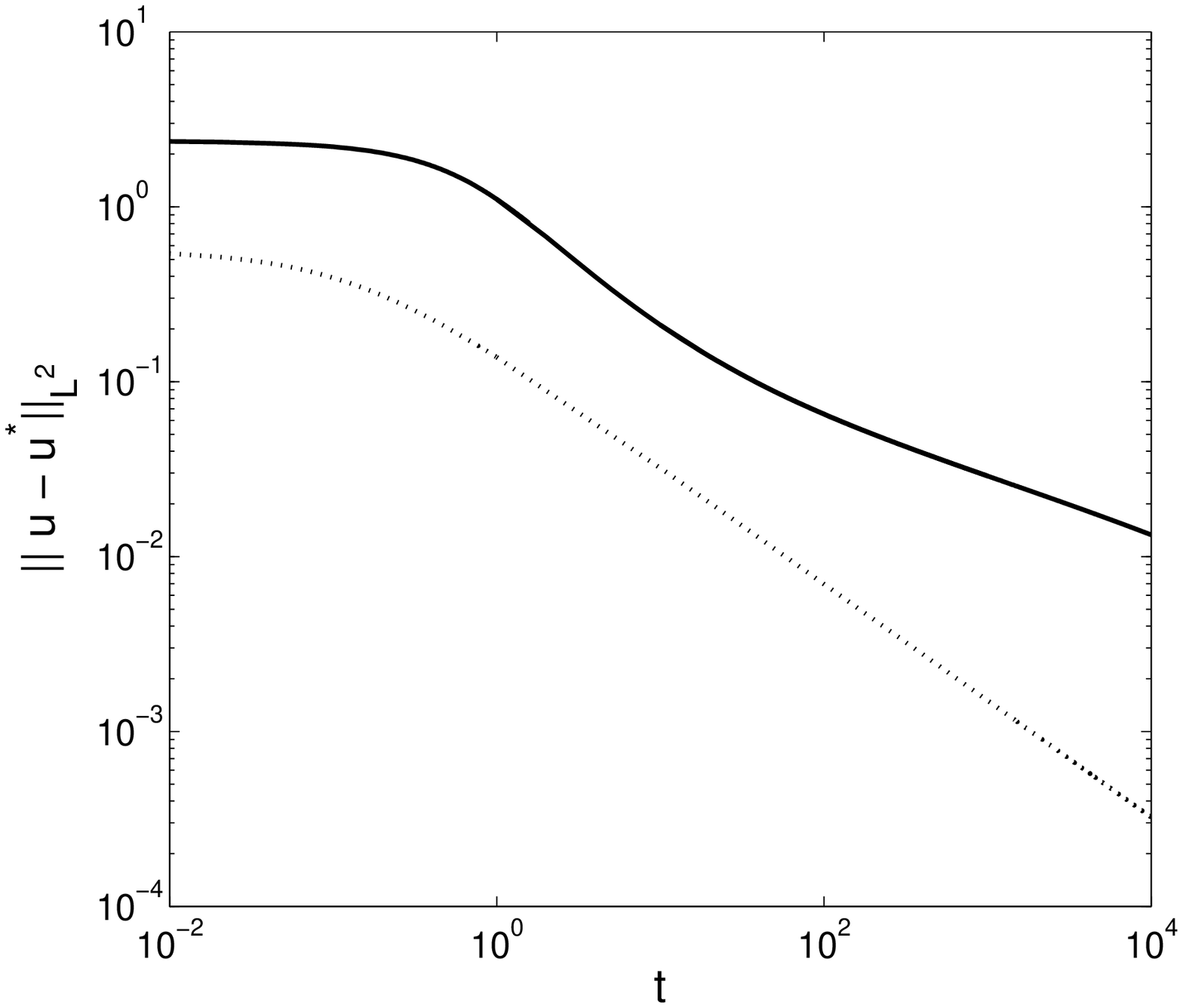}
\end{center}
\caption{\small Evolution of a solution
with $\alpha=1$, $n=3$ and initial data $u_0=1$.
On the left: Time shots of the
numerical solution at $t = 0, 10^{-2}, 10^{-1}, 1, 10, 10^2, 10^3$.
On the right: $L^2$-distance of the solution from the
energy minimizer. The dashed line
shows the bound from Theorem~\ref{thm:decay}.}
\end{figure}

\begin{theorem}[Bounds on the speed of convergence.]
\label{thm:decay}
Consider Eq.~(\ref{eq:PDE}) with $\omega=0$, and set
$\beta=n-\frac{3}{2}$. Let $u$ be a solution
of mass $M$ that dissipates energy and
satisfies Eq.~(\ref{eq:entropy}) with constants $S_0$ and $K_0$,
and let $u^*$  be the energy-minimizing steady state of the
same mass.

\begin{itemize}
\item If $u^*$ is strictly positive,
set $\mu = (1-\alpha^2)(\min u^*)^n$
and $\eps_0=\frac{(1-\alpha^2)\pi}{2}(\min u^*)^2$.
For each solution with $E(u(\cdot, 0))< E(u^*) +\eps_0$ there exists
a constant $K_1$ such that
$$
||u(\cdot, t)-u^*||_{H^1}\le K_1 e^{-\mu t}\,;
$$
\item if $u^*$ vanishes on an interval of positive length $L$,
and $n>\frac{3}{2}$, then
$$
||u(\cdot,t)-u^*||_2\ge L^{1+\frac{\beta}{2}}
\bigl(S_0+K_0t\bigr)^{-\frac{1}{\beta}}\,;
$$
\item
if  $u^*$ vanishes quadratically at a point, and $n>\frac{5}{2}$,
then there exist positive constants $K_2$ and $K_3$ such that
$$
||u(\cdot,t)-u^*||_2\ge \bigl(K_2+K_3t\bigr)^{-\frac{3}{2(\beta-1)}}\,.
$$
\end{itemize}
\end{theorem}

\begin{proof} If $u^*$ is strictly positive,
then $\alpha<1$ by Theorem~\ref{thm:min}, and $E$ is strictly convex.
The Taylor expansion
of $E$ about $u^*$ terminates after the quadratic term, because
$E$ itself is quadratic, and the linear term vanishes because
$u^*$ is a critical point where the positivity constraint is not active,
and therefore
\begin{equation}\label{eq:square}
E(u)=E(u^*) +\frac{1}{2}\int_\Omega (u-u^*)_\theta^2-\alpha^2(u-u^*)^2\,
d\theta\,.
\end{equation}
Since $u$ and $u^*$ have the same mass, the energy difference
dominates the $H^1$-distance,
$$
E(u)-E(u^*)\ge \frac{1-\alpha^2}{2} ||(u-u^*)_\theta||_2^2
\ge \frac{(1-\alpha^2)\pi}{2} ||u-u^*||_\infty^2\,.
$$
Set $\eps= E(u(\cdot, 0))-E(u^*)<\eps_0$,
then $\min u \ge (1\!-\!\frac{\eps}{\eps_0})\min u^*>0$.
This means that $u$ is a strictly positive, classical solution
that can be differentiated as often as necessary.
We compute the $L^2$-gradient of $E$ from Eq.~(\ref{eq:square}) as
$\frac{\delta E}{\delta u}=-(u-u^*)_{\theta\theta}-\alpha^2(u-u^*)$.
By Eq.~(\ref{eq:dissipate}) the energy is dissipated at rate
\begin{eqnarray*}
\frac{d}{dt} E(u(\cdot, t) )
&=&  -\int_\Omega u^n \left [ (u-u^*)_{\theta\theta}+\alpha^2 (u-u^*)
\right]_\theta^2\, d\theta\\
&\le& - (\min u)^n \cdot \int_\Omega
\left [ (u-u^*)_{\theta\theta\theta}+\alpha^2 (u-u^*)_\theta \right]^2
\,d\theta\\
&\le& -2\, \left((1\!-\!\frac{\eps}{\eps_0})\min u^*\right)^n\, (1-\alpha^2)\,
\left( E(u)-E(u^*)\right)\,.
\end{eqnarray*}
In the last step, we have used Parseval's identity
to rewrite the integral  and the energy difference in
terms of the Fourier coefficients of $u-u^*$, and
estimated the Fourier multipliers by
$$
p^2(p^2-\alpha^2)^2\ge (1-\alpha^2)(p^2-\alpha^2)\,,\quad (p\ne 0)\,.
$$
Exponential convergence of the energy follows
from Gronwall's lemma. Since $\min u$ converges exponentially
to $\min u^*$, we conclude that
$E(u(\cdot, t))-E(u^*)\le Ke^{-2\mu t}$ for some constant $K$,
and the first claim follows.

If $u^*$ vanishes on an interval of
length $L>0$, we apply Jensen's inequality to the
convex function $y\mapsto y^{-\frac{\beta}{2}}$ on this
interval to obtain
$$
S(u(\cdot, t)) \ge L^{1+\frac{\beta}{2}}\, ||u(\cdot,t)-u^*||_2^{-\beta}\,.
$$
The second claim follows by using
the bound on the entropy in Eq.~(\ref{eq:entropy})  and solving
for the distance $||u(\cdot, t)-u^*||_2$.

If $u^*$ vanishes quadratically at a point, we consider the interval
of length $L$ centered at that point and obtain with the same
calculation as for the second case that
\begin{eqnarray*}
||u(\cdot, t)-u^*||_2
&\ge & ||u(\cdot, t)I_{|\theta|\ge\tau }||_2
- ||u^*I_{|\theta|>\tau}||_2\\
& \ge& L^{\frac{1}{\beta} + \frac{1}{2}}
\bigl(S_0 + K_0t\bigr)^{-\frac{1}{\beta}} -
O(L^{\frac{3}{2}})\,.
\end{eqnarray*}
The proof is completed by
choosing $L=\eps(S_0+K_0t)^{-\frac{1}{\beta-1}}$ for 
$\eps>0$ sufficiently small.
\end{proof}

\bigskip

Consider specifically the case of Eq.~(\ref{eq:PDE})
where $n=3$, $\alpha=1$, and $\omega=0$.
By Theorem~\ref{thm:min}, the energy minimizer
vanishes on an interval of length $L=2(\pi\!-\!\tau)>0$,
and by Theorem~\ref{thm:decay}, solutions cannot converge to
the minimizer more quickly than $t^{-\frac{2}{3}}$.  We
suspect that the distance from the steady state actually
decays with $t^{-\frac{1}{3}}$.

This conjecture is supported by simulations, and by analogy
with aggregation processes where the convergence to states 
is governed by power laws.  In such processes, the speed 
of convergence is limited by the rate at which mass 
can be transferred from a region of low density to the 
region of accumulation.  One example is the Lifshitz-Slyozov cubic law,
which describes late-stage grain growth in alloys and
the evaporation-condensation mechanism in
supersaturated solutions~\cite{LifSly}.
A second example is the separation of a water drop
from a tap, shown on the left of Figure~5.
The drop slowly draws water from the tap through a 
thin neck.  The right side of Figure~5
shows the accumulation of the mass of a solution
of Eq.~(\ref{eq:PDE}) at the bottom of the cylinder.
The large drop at the bottom grows by
slowly pulling mass from the small drop on the top
through the thin bridges that connect them.

\begin{figure}[ht]
\label{fig:tap}
\begin{center}
\includegraphics[height= 5.5 cm] {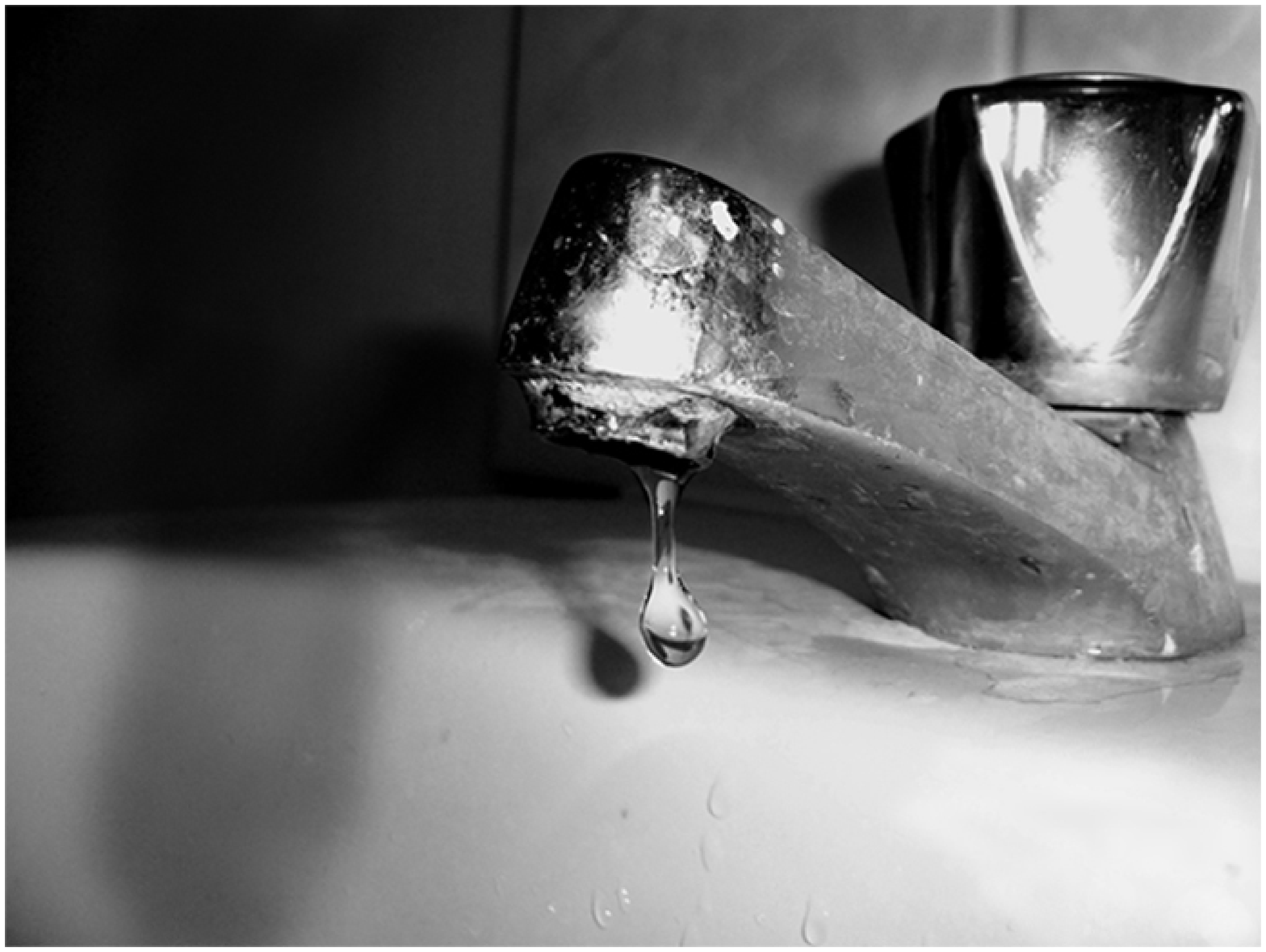}
\quad\quad \quad\qquad
\includegraphics[height= 5.5 cm] {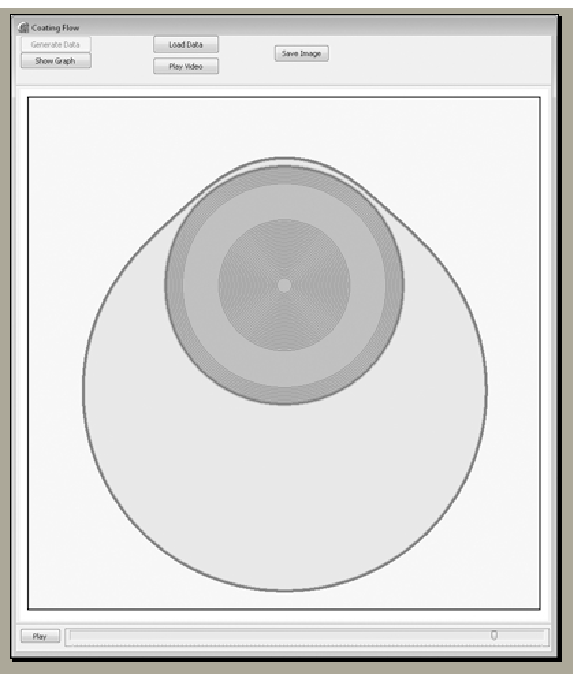}
\end{center}
\caption{\small
On the left: Separation of a water drop at a tap.
Photograph ``Perfect Back and White  ---
Water drop'', by Elliott Minns~\cite{Minns}.
On the right: Matlab simulation of
the long-time behavior of a solution of
Eq.~(\ref{eq:PDE}) with parameters $n=3$, $\alpha=1$,
and  $\omega=0$. The thickness of the film is
exaggerated to emphasize the shape.}
\end{figure}

\subsection* {\bf Acknowledgments} The research in this paper
was partially supported by an NSERC Discovery grant (A.B., B.S.) 
and an NSERC postdoctoral fellowship (M.C.).
We are grateful to Elliott Minns for permission
to reproduce his photograph ``Perfect Black and
White --- Water Drop''~\cite{Minns}.

\end{document}